\documentclass{article}

\newcommand{\Addresses}{{
		\bigskip
		\footnotesize
		
		\textsc{Department of Mathematics, Technion -- Israel Institute of Technology, Haifa, Israel}\par\nopagebreak
		\textit{E-mail address:} \texttt{ofir.gor@technion.ac.il}

}}

\title{A Kubilius model for sieve-theoretic sequences}
\author{Ofir Gorodetsky}
\date{}

\usepackage{amssymb}
\usepackage{amsfonts}
\usepackage{amsmath}
\usepackage{amsthm}
\usepackage{xcolor}
\usepackage{hyperref}

\usepackage[margin=1in]{geometry}

\theoremstyle{plain}
\newtheorem{thm}{Theorem}[section]
\newtheorem{lem}[thm]{Lemma}  
\newtheorem{proposition}[thm]{Proposition}
\newtheorem{cor}[thm]{Corollary}
\theoremstyle{remark}
\newtheorem{rem}{Remark}[section]

\newcommand{\PP}{\mathbb{P}}
\newcommand{\RR}{\mathbb{R}}
\newcommand{\ZZ}{\mathbb{Z}}
\newcommand{\NN}{\mathbb{N}}

\newcommand{\Pa}{\mathcal{P}}
\newcommand{\dtv}{\mathrm{d}_{\mathrm{TV}}}
\newcommand{\Geom}{\mathrm{Geom}}

\numberwithin{equation}{section}

\begin{document}

\maketitle

{\centering To Professor J\'anos Pintz on his 75th birthday\par}
\begin{abstract}
We bound the total variation distance in the Kubilius model for sequences with positive level of distribution. We obtain a result that we expect is qualitatively optimal. As a special case, it recovers a recent result of Ford on shifted primes, with a slightly simplified proof. In the classical case considered by Kubilius, our theorem gives a simple proof of the optimal bound discovered by Tenenbaum, up to factors of $x^{o(1)}$ and $u^{o(u)}$. 
\end{abstract}
\section{Introduction and history}
Given a prime number $p$ and a positive integer $n$, let $\nu_p(n)$ be the exponent of $p$ dividing $n$, i.e.~$p^{\nu_p(n)}$ divides $n$ but $p^{\nu_p(n)+1}$ does not divide $n$. Given discrete random variables $X$ and $Y$ taking values in a countable set $\Omega$, their total variation distance is defined as
\[ \dtv(X,Y) := \sup_{A\subseteq \Omega} |\PP(X\in A)-\PP(Y\in A)|,\]
and is known to satisfy \cite[Eq.~(3.7)]{ABT}
\begin{equation}\label{eq:dtv1} 
	\dtv(X,Y) = \frac{1}{2}\sum_{\omega \in\Omega} |\PP(X=\omega)-\PP(Y=\omega)|.
\end{equation}
\subsection{Kubilius' work}\label{sec:kub}
Let $N_x$ be an integer chosen uniformly at random from $\NN\cap [1,x]$. In Chapter 2 of his book \cite{Kubiliusbook1962} (later translated to English \cite{KubTranslate}),
Kubilius introduced -- and implemented -- the idea of modeling the tuple
\begin{equation}\label{eq:tup1}
(\nu_p(N_x))_{p\le y}
\end{equation}
of dependent random variables by a tuple of independent random variables
\begin{equation}\label{eq:tup2}
(\Geom_p)_{p\le y}
\end{equation}
where $\PP(\Geom_p=k)=p^{-k}(1-1/p)$ for all $k\in \NN_{\ge 0}$. This idea has its origins in Kubilius' 1956 survey paper \cite{Kubpaper} (also translated to English \cite{KubTranslatepaper}). Kubilius studied the total variation distance between the random tuples \eqref{eq:tup1} and \eqref{eq:tup2}. He showed that there exists an absolute constant $c>0$ such that  \cite[p.~27]{KubTranslate}
\begin{equation}\label{eq:tvk}
\dtv( (\nu_p(N_x))_{p\le y}, (\Geom_p)_{p\le y})\ll e^{-cu}
\end{equation}
holds uniformly for $2\le y \le x$ where
\[ u:=\frac{\log x}{\log y}.\]
In particular, $\dtv( (\nu_p(N_x))_{p\le y}, (\Geom_p)_{p\le y})$ goes to $0$ as soon as $u$ tends to infinity.\footnote{This range is known to be optimal, i.e.~the distance does not go to zero if $u$ is bounded  \cite[p.~145]{Elliott1}  \cite[Th\'{e}or\`{e}me 1]{TenenbaumCrible}.} The bound \eqref{eq:tvk} and the model are both known as the \textit{Kubilius model}. The bound  \eqref{eq:tvk}  is a powerful tool in probabilistic number theory, and led Kubilius to definitive results on the distribution of additive functions  \cite[Chapters IV, VI]{KubTranslate}. See also Elliott's treatment of Kubilius' work on additive functions \cite[Chapter 12]{Elliott2} (which includes a short discussion comparing it with the earlier work of Erd\H{o}s and Kac \cite[p.~26]{Elliott2}).

Using a technical result of Barban and Vinogradov \cite{BV}, Kubilius \cite[Equation~(10)]{Brownian} improved \eqref{eq:tvk} significantly to
\begin{equation}\label{eq:tvk2}
	\dtv( (\nu_p(N_x))_{p\le y}, (\Geom_p)_{p\le y})\ll v^{-dv},\qquad v:=\frac{\log x}{\log \max\{y,\log x\}}
\end{equation}
for some absolute constant $d>0$.\footnote{Strictly speaking, Kubilius considered $\mathbf{1}_{p\mid N_x}$ instead of $\nu_p(N_x)$, and a Bernoulli random variable instead of $\Geom_p$.} The constant $d$ has been made explicit by Elliott \cite[pp.~119--123]{Elliott1}, who showed that
\begin{equation}\label{eq:tvbv}
	\dtv( (\nu_p(N_x))_{p\le y}, (\Geom_p)_{p\le y})\ll u^{-u/8}+ x^{-1/15}
\end{equation}
holds uniformly for $2\le y \le x$.
\subsection{Elliott's generalization}\label{sec:ellgen}
The arguments of Kubilius are sieve-theoretic and are quite flexible. Several variants of \eqref{eq:tvbv} were derived by adapting the argument of Kubilius, and most of them are discussed in Chapter 3 of Elliott \cite{Elliott1}. The chapter contains  rich material on the history of the topic and interesting concluding remarks. Let $h\in \ZZ[t]$. We mention two variants of \eqref{eq:tvbv}:
\begin{itemize}
	\item  During the period 1959--1967, U\v{z}davinis wrote several papers concerning additive functions evaluated on $h(n)$. Let  $(W_p)_{p_0<p\le y}$ be independent  Bernoulli random variables with $\PP(W_p=1)=\rho(p)/p$ where $\rho$ is the number of distinct roots of $h$ modulo $p$; here $p_0$ is chosen sufficiently large so that $\rho(p)<p$ holds for $p>p_0$. U\v{z}davinis \cite{Uzd} implicitly bounded the total variation distance between $(\mathbf{1}_{p\mid h(N_x)})_{p_0<p\le y}$ and $(W_p)_{p_0<p\le y}$. His result is stated and  discussed in Kubilius' book \cite[p.~28]{KubTranslate}. We state the result with the error term obtained by Elliott \cite[pp.~131--134]{Elliott1}, using the input from \cite{BV}:
	\begin{equation}\label{eq:Uzdavinis}
		\dtv( (\mathbf{1}_{p\mid h(N_x)})_{p_0<p\le y}, (W_p)_{p_0<p\le y})\ll_h  u^{-u/50}+x^{-1/4}.
	\end{equation}
	\item Let $P_{x}$ be a prime number chosen uniformly at random from the set of primes up to $x$. Let  $(\widetilde{W}_p)_{p_0<p\le y}$ be independent  Bernoulli random variables with $\PP(\widetilde{W}_p=1)=\eta(p)/(p-1)$ where $\eta$ is the number of distinct roots of $h$ modulo $p$  \textit{excluding $0\bmod p$}; here $p_0$ is chosen sufficiently large so that $\eta(p)<p-1$ holds for $p>p_0$. Barban \cite{Barban} established a Kubilius model for the factorization of $h(P_x)$, which Elliott \cite[pp.~134--136]{Elliott1} improved as follows using \cite{BV}: for every $B>0$,
	\begin{equation}\label{eq:Barban}
		\dtv( (\mathbf{1}_{p\mid h(P_x)})_{p_0<p\le y}, (\widetilde{W}_p)_{p_0<p\le y})\ll_{B,h} u^{-u/75}+(\log x)^{-B}.
	\end{equation}
\end{itemize}
The bounds \eqref{eq:Uzdavinis}--\eqref{eq:Barban} are special cases of the following general result of Elliott. We need the following notation: given a subset $\Pa$ of primes, let $\langle \Pa \rangle$ be the set of positive integers divisible only by those primes, i.e.~the semigroup generated by $\Pa$. 
\begin{thm}\cite[pp.~129--132]{Elliott1}\label{thm:Kub}
Let $A$ be a random variable  taking values in $[1,x]\cap \NN$. Let $y\in [2,x]$. Let $\Pa$ be a subset of the primes up to $y$. Let $g\colon \NN\to [0,1]$ be a multiplicative function with $g(d)=0$ for all $d\not\in \langle \Pa\rangle$ and  $g(p)<1$ for all $p\in \Pa$. Let $(B_p)_{p\in \Pa}$ be independent Bernoulli random variables with $\PP(B_p=1)=g(p)$. 	Define
\[ S=\sum_{p\le y}\frac{g(p)}{1-g(p)}\log p.\] 
Then, 	for any parameter $z$ that satisfies $\log z\ge 8\max\{\log y,S\}$, we have
	\[	\dtv( (\mathbf{1}_{p\mid A})_{p\in \Pa}, (B_p)_{p \in \Pa}) \le 10 \exp\bigg( -\frac{\log z}{8\log y}\log\bigg( \frac{\log z}{S}\bigg)\bigg)+12\sum_{\substack{m\le z^4 \\m\in \langle \Pa\rangle}}\mu^2(m)4^{\omega(m)}|\PP(m \mid A)-g(m)|.\]
	\end{thm}
	In the statement of Theorem \ref{thm:Kub}, $\PP(d\mid A)$ is the probability that $A$ is divisible by $d$. To illustrate Theorem \ref{thm:Kub}, we apply it in a special case. Suppose $A=N_x$ and that $\Pa$ is the set of all primes up to $y$. Let $g(d)=1/d$ for $d\in \langle \Pa \rangle$ so that $S=\log y + O(1)$ by Mertens' theorem. Since $\PP(d\mid N_x)-1/d \ll 1/x$, Theorem \ref{thm:Kub} implies that
\[	\dtv( (\mathbf{1}_{p\mid N_x})_{p\le y}, (B_p)_{p \le y}) \ll \exp\left( -\frac{\log z}{8\log y}\log\left( \frac{\log z}{S}\right)\right)+x^{-1}\sum_{\substack{m\le z^4\\ m\in \langle \Pa \rangle}}\mu^2(m)4^{\omega(m)}\]
holds for any $x\ge y \ge 2$ and any $z\ge Cy^8$. Taking $z=x^{1/5}$ (which is a valid choice if $u$ is sufficiently large) yields the bound $u^{-u/50}+x^{-1/6}$.
\subsection{Tenenbaum's bound}
Utilizing complex analysis, Tenenbaum \cite[Th\'{e}or\`{e}me 1]{TenenbaumCrible} obtained a further improvement to \eqref{eq:tvbv}, which in some ranges of $x$ and $y$ even gives an asymptotic formula for the total variation distance. While his result is technical to state,  it implies  the following simple bound \cite[Equation~(1.7)]{TenenbaumCrible}: for every $\varepsilon>0$,
\begin{equation}\label{eq:tenen}
	\dtv( (\nu_p(N_x))_{p\le y}, (\Geom_p)_{p\le y})\ll_{\varepsilon}  u^{-u}+x^{-1+\varepsilon}.
\end{equation}
\subsection{Ford's result on shifted primes}\label{sec:ford}
Fix $a\in \ZZ\setminus \{0\}$. Ford \cite{FordShifted} established a sharp analogue of the Kubilius model for shifted primes, which improves on \eqref{eq:Barban} when $h(t)=t+a$. To state it, let $(\widetilde{G}_p)_{p \le y}$ be independent random variables with \[\PP(\widetilde{G}_p=k)=\frac{1}{\phi(p^k)}-\frac{1}{\phi(p^{k+1})}=\begin{cases} 1-\frac{1}{p-1} &\text{if }k=0,\\p^{-k}&\text{if }k\ge 1,\end{cases}\]
if $p\nmid a$, and $\widetilde{G}_p\equiv 0$ if $p\mid a$. Suppose that $\gamma\in (0,1]$ is such that the following hypothesis holds:
\begin{equation}\label{eq:Hz}
\forall \varepsilon>0\, \exists \delta>0\, \forall B>0:  \sum_{\substack{m\le x^{\gamma-\varepsilon}\\P(m)\le x^{\delta}\\(m,a)=1}}\bigg|\pi(x;m,-a)-\frac{\pi(x)}{\phi(m)}\bigg| \ll_{B,\varepsilon,a} \frac{x}{(\log x)^B}
\end{equation}
where $\pi(x;m,-a)$ is the number of primes up to $x$ congruent to $-a\bmod m$, and  $P(m)$ is the largest prime divisor of $m$ (with $P(1):=1$). Let $P_{x,a}$ be a prime chosen uniformly at random from the primes in $(|a|+1,x]$. Then one has the bound \cite[Theorem~1]{FordShifted}
\begin{equation}\label{eq:ford}
\dtv( (\nu_p(P_{x,a}+a))_{p\le y},(\widetilde{G}_p)_{p\le y} )\ll_{a,B,\theta} u^{-\theta u} + (\log x)^{-B}
\end{equation}
for every $\theta \in (0,\gamma)$ and  every $B>0$, uniformly for $2\le y \le x$. One may take $\gamma=1/2$ due to the Bombieri--Vinogradov theorem. The exponent $1/2$ should be compared with the exponent $1/75$ in \eqref{eq:Barban}.
\subsection{Permutations}
A Kubilius model for permutations was studied by Arratia and Tavar{\'e} \cite[Theorem 2]{AT}, and their result is a perfect analogue of \eqref{eq:tenen}.  Ford \cite[Theorem 1.19]{Toolkit} obtained a slightly weaker result than Arratia and Tavar{\'e}'s, but with a considerably simpler proof, inspired by the work of Kubilius. Manstavi{\v{c}}ius and Petuchovas 
\cite[Theorem 3]{MP} adapted the main result of Tenenbaum \cite{TenenbaumCrible} to permutations (cf.~\cite{petuchovasthesis}). For an exposition of the Kubilius model for permutations and integers, see Chapter 4 of Ford's notes \cite{FordNotesAnatomy}.
\subsection*{Conventions}
Throughout the paper, the implicit constants in bounds of the shape $A\ll B$ or $A=O(B)$ are absolute, unless indicated otherwise using subscripts.
\section{Main result}
Our main theorem is a sharpening of Theorem \ref{thm:Kub}. Our proof builds on the sieve-theoretic approach laid out by Kubilius and generalized by Elliott, but also borrows the clever choice of parameters from Ford's work on shifted primes \cite{FordShifted}, and the particular fundamental lemma of sieve theory Ford used. 

Informally, our result will say the following: let $A$ be a random variable taking values in $[1,x]\cap \NN$. Let $g$ be a multiplicative function taking values in $[0,1]$, such that $g(p)\sim \kappa/p$ holds `on average' for some $\kappa>0$. Define $(G_p)_p$ to be independent random variables such that $\PP(G_p=k)=g(p^k)-g(p^{k+1})$ for every $k\in \NN_{\ge 0}$. Then, roughly speaking, we shall show that
	\[	\dtv( (\nu_p(A))_{p\le y}, (G_p)_{p \le y})\, ``\ll \,"  u_T^{-u_T}+ \sum_{\substack{d\le T\\ P(d)\le y}} |\PP(d\mid A)-g(d)|\]
	holds, where $T$ is a parameter at one's disposal and $u_T:=\log T/\log y$. In comparison, Theorem \ref{thm:Kub} essentially gives the weaker bound
	\[	\dtv( (\nu_p(A))_{p\le y}, (G_p)_{p \le y})\, ``\ll \," u_T^{-u_T}+ \sum_{\substack{d\le T^{32}\\ P(d)\le y}} |\PP(d\mid A)-g(d)|\]
	if one replaces Elliott's $z$ with $T^{8}$. Before stating the full result we give two applications. 
\begin{cor}\label{cor:giden}
	Let $A$ be a random variable taking values in $[1,x]\cap \NN$. Let $y\in [2,x]$. Let $(\Geom_p)_{p\le y}$ be independent random variables defined as in \S\ref{sec:kub}. For every $\theta \in (0,1)$ and $\varepsilon>0$, we have
	\begin{equation}\label{eq:dtvgeom}
	\dtv( (\nu_p(A))_{p\le y}, (\Geom_p)_{p \le y}) \ll_{\theta,\varepsilon} u^{-\theta u(1-\varepsilon)}+x^{-\theta(1-\varepsilon)}+ x^{1/\log \log x}\sum_{\substack{d\le x^{\theta}\\ P(d)\le y}} |\PP(d\mid A)-1/d|.
\end{equation}
\end{cor}
Let $\Psi(x,y)$ be the number of integers $n \in [1,x]$ with $P(n)\le y$. Recall $\Psi(x,y)\ll_{\varepsilon} x u^{-u} + x^{\varepsilon}$ \cite[Ch.~III.5, Exercise 293(c)]{Ten}. Applying Corollary \ref{cor:giden} with $A=N_x$ and $\theta=1-\varepsilon$, and observing $\PP(d\mid A)-1/d \ll 1/x$, we obtain
\begin{equation}\label{eq:tenen2}
	\dtv( (\nu_p(N_x))_{p\le y}, (\Geom_p)_{p\le y})\ll_{\varepsilon}  u^{-u(1-\varepsilon)}+x^{-1+\varepsilon}
\end{equation}
for every $\varepsilon>0$, which is a slightly weaker version of Tenenbaum's optimal bound \eqref{eq:tenen}. Our results  allow one to replace $\varepsilon$ by an explicit quantity tending to $0$. A slightly more technical corollary is as follows.
\begin{cor}\label{cor:gphi}
	Let $A$ be a random variable taking values in $[1,x]\cap \NN$. Let $y\in [2,x]$. Fix $a\in \ZZ\setminus \{0\}$. Let $(\widetilde{G}_p)_{p\le y}$ be independent random variables defined as in \S\ref{sec:ford}.
	For every $\theta \in (0,1)$ and $\varepsilon>0$, we have
	\begin{multline*}
		\dtv( (\nu_p(A))_{p\le y}, (\widetilde{G}_p)_{p \le y}) \ll_{a,\varepsilon,\theta}u^{-\theta u(1-\varepsilon)}+x^{-\theta(1-\varepsilon)} \\
		+\sum_{\substack{d\le x^{\theta}\\ P(d)\le y\\ (d,a)=1}} \big|\PP(d\mid A)-\frac{1}{\phi(d)}\big|(2^{\omega(d)}+\log y)+  \sum_{\substack{d\le x^{\theta}\\ P(d)\le y\\ (d,a)>1}} \PP(d\mid A)(2^{\omega(d)}+\log y).
	\end{multline*}
\end{cor}
We apply Corollary \ref{cor:gphi} with $A=P_{x,a}+a$, where $P_{x,a}$ is a prime chosen uniformly at random from the primes in $(|a|+1,x]$. If $a>0$ we must apply the corollary with $x+a$ instead of $x$ because $A$ can be as large as $x+a$. 	By definition,
\[\PP(d\mid P_{x,a}+a) = \frac{\pi(x;d,-a)-\pi(|a|+1;d,-a)}{\pi(x)-\pi(|a|+1)}= \frac{\pi(x;d,-a)}{\pi(x)}+O_a\left(\frac{1}{\pi(x)}\right).\]
Since $\pi(x;d,-a)\le 1$ if $(d,a)>1$, we find that
	\[	\dtv( (\nu_p(P_{x,a}+a))_{p\le y}, (\widetilde{G}_p)_{p \le y}) \ll_{a,\varepsilon,\theta}u^{-\theta u(1-\varepsilon)}+x^{-\theta(1-\varepsilon)} +\sum_{\substack{d\le (x+|a|)^{\theta}\\ P(d)\le y\\ (d,a)=1}} \bigg|\frac{\pi(x;d,-a)}{\pi(x)}-\frac{1}{\phi(d)}\bigg|(2^{\omega(d)}+\log y)\]
	holds by invoking  $\Psi(x,y)\ll_{\varepsilon} xu^{-u} +x^{\varepsilon}$ again to treat the contribution of $O_a(1/\pi(x))$ and $(d,a)>1$. To recover Ford's result on shifted primes, namely the bound \eqref{eq:ford}, we apply the Cauchy--Schwarz inequality and the Brun--Titchmarsh theorem to conclude
		\[ \sum_{\substack{d\le (x+|a|)^{\theta}\\P(d)\le y\\(d,a)=1}} \left|\frac{\pi(x;d,-a)}{\pi(x)} - \frac{1}{\phi(d)}\right|(\log y+ 2^{\omega(d)}) \ll_{a,\theta}(\log y)^2 \bigg(\sum_{\substack{d\le (x+|a|)^{\theta}\\P(d)\le y\\(d,a)=1}} \left|\frac{\pi(x;d,-a)}{\pi(x)} - \frac{1}{\phi(d)}\right|\bigg)^{1/2},\]
		which can be estimated by Ford's assumption \eqref{eq:Hz}. Next we state our full theorem.    Recall that $\langle \Pa \rangle$ is the semigroup generated by $\Pa$.
\begin{thm}\label{thm:KM}
	Consider the following input:
	\begin{itemize}
		\item Let $A$ be a random variable  taking values in $[1,x]\cap \NN$. 
		\item Let $y\in [2,x]$.
			\item Let $\Pa$ be a subset of the primes up to $y$.
		\item Let $g\colon \NN\to [0,1]$ be a multiplicative function with $g(d)=0$ for all $d\not\in \langle \Pa\rangle$. Suppose that $g(p^{k+1})\le g(p^k)$	for all $p\in \Pa$ and $k\in \NN_{\ge 0}$, and that $\lim_{k\to \infty}g(p^k)=0$ for all $p\in \Pa$. Suppose that there exist $\kappa\ge 0$ and $B>0$ such that the following holds: for every $d\in \langle\Pa\rangle$ such that $g(p^{\nu_p(d)+1})<g(p^{\nu_p(d)})$ for all $p \in \Pa$, we have
		\begin{equation}\label{eq:iwaniecd}
			\prod_{a\le p \le b} \left( 1-\frac{g(pd)}{g(d)}\right)^{-1} \le \left( \frac{\log b}{\log a}\right)^{\kappa}\exp\left( \frac{B}{\log a}\right)
		\end{equation}
	for all $2\le a \le b\le y$.
		\item Let $(G_p)_{p\in \Pa}$ be independent random variables taking values in $\NN_{\ge 0}$ with $\PP(G_p=k)=g(p^k)-g(p^{k+1})$ for all $k\in \NN_{\ge 0}$.
	\end{itemize}
	Then, 	for any parameters $S>T>0$ with $S/T\ge y$, we have
	\[	\dtv( (\nu_p(A))_{p\in\Pa}, (G_p)_{p\in \Pa}) \ll_{\kappa,B} D_1+D_2\]
	for
	\begin{align*}
		D_1 &=\sum_{\substack{ d>T\\ g(d)\neq 0}} g(d)\prod_{p\in \Pa}(1-g(p^{\nu_p(d)+1})/g(p^{\nu_p(d)}))+\sum_{\substack{ d \le T\\ g(d)\neq 0}} g(d) \prod_{p\in \Pa}(1-g(p^{\nu_p(d)+1})/g(p^{\nu_p(d)})) f(u_d),\\
		D_2 &=  \sum_{\substack{d\le S\\ d \in \langle \Pa\rangle }} |\PP(d\mid A)-g(d)|(2^{\omega(d)}+\mathbf{1}_{g(d)\neq 0}\prod_{p\in \Pa}(1+g(p^{\nu_p(d)+1})/g(p^{\nu_p(d)}))),
	\end{align*}
	where $u_d$ is given by $y^{u_d} := S/d$, and $f(t):=e^{-t\log t+t\log_3 \max\{t,100\}+C_{\kappa,B}t}$ for a sufficiently large constant  $C_{\kappa,B}$. Here $\log_3 t=\log \log \log t$.
\end{thm}
The error term $D_2$ encodes information on the level of distribution of the random variable $A$.
The error term $D_1$ does not depend on $A$ at all; the dependence on $A$ comes into play only in $D_2$.

In \S\ref{sec:rankin} we discuss how one bounds $D_1$. For reasonable functions $g$ we show that $D_1\approx u_T^{-u_T}$, so $D_1$ decreases in $T$. Clearly, $D_2$ increases in $S$. In view of the restriction $S\ge Ty$, we shall usually take $S=Ty$.
\subsection{Bernoulli variant}
From Theorem \ref{thm:KM} one can deduce the following usable variant, where $(B_p)_{p \in \mathcal{P}}$ is defined as in Theorem \ref{thm:Kub}.
\begin{cor}\label{cor:KM}
	Let $A$ be a random variable  taking values in $[1,x]\cap \NN$. Let $y\in [2,x]$. Let $\Pa$ be a subset of the primes up to $y$. Let $g\colon \NN\to [0,1]$ be a multiplicative function supported on the squarefree elements in $\langle \Pa \rangle$, and such that $g(p)<1$ for all $p\in \Pa$. Suppose that there exist $\kappa\ge 0$ and $B>0$ such that
	\begin{equation}\label{eq:iwaniecdd}
		\prod_{a\le p \le b} \left( 1-g(p)\right)^{-1} \le \left( \frac{\log b}{\log a}\right)^{\kappa}\exp\left( \frac{B}{\log a}\right)
	\end{equation}
	holds for all $2\le a \le b\le y$. Then, 	for any parameter $T>0$, we have
	\[	\dtv( (\mathbf{1}_{p\mid A})_{p\in\Pa}, (B_p)_{p\in \Pa}) \ll_{\kappa,B} D'_1+D'_2\]
	for
	\begin{align*}
		D'_1 &=\prod_{p\in \Pa}(1-g(p))\bigg(\sum_{d>T}\frac{g(d)}{\prod_{p\mid d}(1-g(p))}+\sum_{ d \le T} \frac{g(d)}{\prod_{p\mid d}(1-g(p))}f(u_d)\bigg),\\
		D'_2 &=  \sum_{\substack{d\le Ty\\ d \in \langle \Pa\rangle \\ \mu^2(d)=1}} |\PP(d\mid A)-g(d)|(2^{\omega(d)}+\prod_{p\in \Pa}(1+g(p))),
	\end{align*}
	where $u_d$ is given by $y^{u_d} := Ty/d$, and $f(t):=e^{-t\log t+t\log_3 \max\{t,100\}+C_{\kappa,B}t}$ for a sufficiently large   $C_{\kappa,B}$.
\end{cor}
\begin{proof}
	Apply Theorem \ref{thm:KM} to the random variable $A'= \prod_{p \mid A}p$. Then take $S=Ty$.
\end{proof}
Via Corollary \ref{cor:KM} we shall improve the estimates of U\v{z}davinis and Barban discussed in the introduction.
\begin{cor}\label{cor:bern}
Fix $\varepsilon>0$ and $B>0$. In the notation of \S\ref{sec:ellgen}, the following estimates hold:
\begin{align}\label{eq:Uzdavinisimp}
	\dtv( (\mathbf{1}_{p\mid h(N_x)})_{p_0<p\le y}, (W_p)_{p_0<p\le y})&\ll_{h,\varepsilon}  u^{-u(1-\varepsilon)}+x^{-1+\varepsilon},\\
\label{eq:Barbanimp}
	\dtv( (\mathbf{1}_{p\mid h(P_x)})_{p_0<p\le y}, (\widetilde{W}_p)_{p_0<p\le y})&\ll_{B,h,\varepsilon} u^{-u(1/2-\varepsilon)}+(\log x)^{-B}.
\end{align}
\end{cor}
\subsection{Remark on the proof of Theorem \ref{thm:KM}}
Ford's starting point for proving \eqref{eq:ford} is the formula \eqref{eq:dtv1} for the total variation distance of two discrete random variables $X$ and $Y$. If one wants to bound the contribution of $\omega \in \Omega_{\mathrm{bad}}$ to \eqref{eq:dtv1} for some subset $\Omega_{\mathrm{bad}}\subseteq \Omega$, the triangle inequality bounds this trivially by
\[ \frac{1}{2}(\PP(X \in \Omega_{\mathrm{bad}}) +\PP(Y \in \Omega_{\mathrm{bad}})),\]
but in many situations $\PP(X \in \Omega_{\mathrm{bad}})$  is easy to handle, while $\PP(Y \in \Omega_{\mathrm{bad}})$ is difficult to bound. We avoid \eqref{eq:dtv1}, and instead use the following (also well-known) formula \cite[Eq.~(3.7)]{ABT}:
\begin{equation}\label{eq:dtvmax}
	\dtv(X,Y)= \sum_{\omega\in \Omega} \max\left\{\PP(X=\omega)-\PP(Y=\omega),0\right\}.
\end{equation}
Identity \eqref{eq:dtvmax} follows from the observation that $|\PP(X\in A)-\PP(Y\in A)|$ is maximized when $A=\{ \omega\in \Omega: \PP(X=\omega)\ge \PP(Y=\omega)\}$. Formula \eqref{eq:dtvmax} has the advantage that $\omega\in \Omega_{\mathrm{bad}}$ contributes to it at most
\[ \PP(X \in \Omega_{\mathrm{bad}}).\]
In particular, this leads to our proof of Ford's bound requiring less arithmetic input (Ford needed bounds on the probability that $P_{x,a}+a$ is $y$-smooth while we do not).
The observation that \eqref{eq:dtvmax} is superior to \eqref{eq:dtv1} in the context of the Kubilius model is implicit in Elliott \cite[pp.~123--126]{Elliott1}. The first explicit use of \eqref{eq:dtvmax} in the context of the Kubilius model seems to be Ford's lecture notes \cite[Ch.~4]{FordNotesAnatomy}, where \eqref{eq:dtvmax} plays a crucial role.
\section{Sieve theory facts}
Before proving Theorem \ref{thm:KM}, we review classical results from sieve theory that we will require. The following proposition is a form of the fundamental lemma of sieve theory. We quote it as stated in Ford's lecture notes \cite{FordNotes}. Its proof in the notes is based on Hooley's `almost pure sieve' \cite{Hooley}, a powerful variant of Brun's pure sieve, as simplified and presented by Ford and Halberstam \cite{FordHalberstam}. Ford and Halberstam refer to this sieve as the Brun--Hooley sieve.

A function $\lambda\colon \NN\to \RR$ is called an \textit{upper bound} (resp.~\textit{lower bound}) sieve if $\lambda(1)=1$ and $\sum_{d\mid n}\lambda(d)\ge \delta_{n,1}$ (resp.~$\sum_{d\mid n}\lambda(d)\le \delta_{n,1}$) for all $n\in \NN$. Here $\delta_{n,1}$ is $1$ if $n=1$, and is $0$ otherwise.

\begin{proposition}\cite[Theorem 3.6(a)]{FordNotes}\label{prop:sieve}
Given integers $z,D$ with $2\le z\le D^{1/2}$, let $s=\tfrac{\log D}{\log z}$ and assume $s\ge 100$. There exist upper bound and lower bound sieves $\lambda^+,\lambda^-\colon \NN\to \RR$ satisfying $|\lambda^+(n)|, |\lambda^-(n)|\le 1$ for all $n\in \NN$, vanishing outside the finite set $\{1 \le d \le D:\, \mu^2(d)=1,\, P(d)\le z\}$, such that the following holds. Given $\kappa\ge 0$ and $B>0$ there is a constant $C_{\kappa,B}>0$ with the following property.  If $g\colon \NN\to [0,1]$ is a multiplicative function satisfying $g(p)<1$ for all $p\le z$, and
\begin{equation}\label{eq:iwanieccond}
\prod_{y\le p\le w}(1-g(p))^{-1} \le \left( \frac{\log w}{\log y}\right)^{\kappa} \exp\left( \frac{B}{\log y}\right)
\end{equation}
for all $2\le y \le w \le z$, then
\[ \sum_{d\ge 1} \lambda^{\pm}(d) g(d)=\left(1+O\left(e^{-s\log s+s\log_3 s+C_{\kappa,B}s}\right)\right)\prod_{p\le z}(1-g(p)).\]
\end{proposition}
\begin{cor}\cite[Theorem 3.6(b)]{FordNotes}\label{cor:fl}
Let $z,D,s,g,\kappa,B,C_{\kappa,B}$ be as in Proposition \ref{prop:sieve}. Let $\Pa$ be a set of primes and suppose $g(d)=0$ if $d\not\in\langle \Pa\rangle$. Let $(a_n)_{n\ge 1}$ be a summable nonnegative sequence. Let $X>0$. For any $d \in \langle \Pa\rangle$, let
\[ r_d:= \sum_{\substack{n:\, d \mid n}} a_n - g(d)X.\] 
Denote $\Pa(z) = \prod_{p\le z,\, p \in \Pa}p$. For some $|\Delta|\le 1$, we have
\begin{equation}\label{eq:flcor}
\sum_{n:\,  \gcd(n,\Pa(z))=1}a_n = X\prod_{p\le z} (1-g(p))(1+O(e^{-s\log s + s\log_3 s +C_{\kappa,B}s}))+\Delta\sum_{\substack{d\le D\\ d\mid \Pa(z)}}|r_d|.
\end{equation}
\end{cor}
\begin{proof}
For any upper bound sieve $\lambda$ we have
\begin{align*}
\sum_{n: \, \gcd(n,\Pa(z))=1} a_n&=\sum_{n} a_n \delta_{\gcd(n,\Pa(z)),1}\le \sum_{n} a_n \sum_{d\mid \gcd(n,\Pa(z))} \lambda(d) \\
&= \sum_{d \mid \Pa(z)}\lambda(d) \sum_{n:\, d\mid n}a_n= \sum_{d \mid \Pa(z)}\lambda(d) \left( g(d) X + r_d\right).
\end{align*}
Specializing $\lambda$ to the upper bound sieve $\lambda^+$ from Proposition \ref{prop:sieve}, we obtain the upper bound part of \eqref{eq:flcor}, i.e.~$\Delta\le 1$. The proof of the lower bound is similar.
\end{proof}
We can restate Corollary \ref{cor:fl} probabilistically.
\begin{cor}\label{cor:flprob}
Let $z,D,s,g,\kappa,B,C_{\kappa,B}$ be as in Proposition \ref{prop:sieve}. Let $\Pa$ be a set of primes and suppose $g(d)=0$ if $d\not\in\langle \Pa\rangle$. Denote $\Pa(z) = \prod_{p\le z,\, p \in \Pa}p$. Let $R$ be a  random variable taking values in $\NN$. For some $|\Delta|\le 1$ we have \[ \PP( \gcd(R,\Pa(z))=1) = \prod_{p \le z}(1-g(p)) (1+O(e^{-s\log s + s\log_3 s+C_{\kappa,B}s})) + \Delta \sum_{\substack{d\le D\\d \mid \Pa(z)}} \left|\PP(d\mid R)-g(d)\right|.\]
\end{cor}
\begin{proof}
Apply Corollary \ref{cor:fl} with $a_n = \PP(R=n)$ and $X=1$.
\end{proof}
When $s$ is small we complement Corollary \ref{cor:flprob} with the following upper bound. It is a probabilistic restatement of Theorem~2.4 in \cite{FordNotes}, which is again based on the Brun--Hooley sieve.
\begin{lem}\cite[Theorem 2.4]{FordNotes}\label{lem:upperprob}
Let $z\ge 2$. Let $g\colon \NN\to [0,1]$ be a multiplicative function satisfying \eqref{eq:iwanieccond} for some $\kappa\ge 0$ and $B>0$, for all $2\le y \le w \le z$. Let $\Pa$ be a set of primes and suppose $g(d)=0$ if $d\not\in\langle \Pa\rangle$. Denote $\Pa(z) = \prod_{p\le z,\, p \in \Pa}p$. Let $R$ be a  random variable taking values in $\NN$. Then
\[ \PP( \gcd(R,\Pa(z))=1) \le O_{\kappa,B}\bigg( \prod_{p \le z}(1-g(p))\bigg) + \sum_{\substack{d\le z\\d \in \langle \Pa\rangle}}\mu^2(d) \left|\PP(d\mid R)-g(d)\right|.\]
\end{lem}
\section{Proof of Theorem \ref{thm:KM}}
\begin{lem}\label{lem:tv}
	Under the notation and assumptions of Theorem \ref{thm:KM},
\[		\dtv( (\nu_p(A))_{p \in \Pa}, (G_p)_{p \in \Pa})=\sum_{\substack{g(d)\neq 0}} \max\bigg\{  \prod_{p\in \Pa}(g(p^{\nu_p(d)})-g(p^{\nu_p(d)+1}))-\PP(d\mid A,\, \gcd(A/d,\prod_{p \in \Pa}p)=1),0\bigg\}.\]
\end{lem}
\begin{proof}
Consider the map that takes a sequence of nonnegative integers $\mathbf{a}=(a_p)_{p\in \Pa}$ to the positive integer $d=d(\mathbf{a}):=\prod_{p\in \Pa} p^{a_p}$. Note that $\mathbf{a}\mapsto d$ is a bijection from the set of sequences of nonnegative integers indexed by primes $p\in \Pa$ to the set of positive integers $d$ with $d\in \langle \Pa\rangle$. The inverse of the bijection takes $d$ to $(\nu_p(d))_{p\in \Pa}$. In this notation,
	\begin{align*}
		\PP((\nu_p(A))_{p\in \Pa}=\mathbf{a}) &= \PP(d \mid A,\, \gcd(A/d,\prod_{p\in \Pa}p)=1),\\
		\PP((G_p)_{p\in \Pa}=\mathbf{a}) &= \prod_{p\in \Pa} \PP(G_p = a_p)=\prod_{p\in \Pa} (g(p^{\nu_p(d)})-g(p^{\nu_p(d)+1})).
	\end{align*}
	The lemma now follows from \eqref{eq:dtvmax} with $X=(G_p)_{p\in \Pa}$ and $Y=(\nu_p(A))_{p\in \Pa}$, except that the $d$-sum we obtain has, instead of the condition $g(d)\neq 0$, the condition $d\in \langle \Pa\rangle$. We may replace $d\in \langle \Pa \rangle$ with $g(d)\neq 0$ due to two observations: first, $g(d)\neq 0$ implies $d\in \langle \Pa\rangle$. Second, if  $d\in\langle \Pa \rangle$ and $g(d)=0$ both hold then the corresponding $d$-th summand is $0$ and can be discarded.
\end{proof}
We use the formula for the total variation distance as given in Lemma \ref{lem:tv}. The range $d>T$ is handled using the trivial bound $\max\{a-b,0\}\le a$, obtaining the bound
\[\le\sum_{\substack{d> T\\ g(d)\neq 0}} g(d) \prod_{p\in \Pa}(1-g(p^{\nu_p(d)+1})/g(p^{\nu_p(d)}))\]
for the contribution of this range, which can be absorbed in $D_1$. It remains to handle the range $d\le T$. Let $A_d$ be the random variable constructed by dividing $A$ by $d$, conditioned on the event $d\mid A$. Then in view of Lemma \ref{lem:tv} we aim to upper bound
\begin{equation}\label{eq:aim}
	\sum_{\substack{d\le T\\g(d)\neq 0}} \max\bigg\{   \prod_{p\in \Pa}(g(p^{\nu_p(d)})-g(p^{\nu_p(d)+1})) - \PP(d\mid A)\PP( \gcd(A_d,\prod_{p \in \Pa}p)=1),0\bigg\}.
\end{equation}
(If $\PP(d\mid A)=0$ we  define $\PP( \gcd(A_d,\prod_{p \in \Pa}p)=1)=0$ formally). When $g(d)\neq 0$ we may write
\begin{align*}
	&\prod_{p\in \Pa}(g(p^{\nu_p(d)})-g(p^{\nu_p(d)+1}))-\PP(d\mid A)\PP( \gcd(A_d,\prod_{p \in \Pa}p)=1) \\
	&=(g(d) -\PP(d\mid A))\prod_{p\in \Pa}(1-g(p^{\nu_p(d)+1})/g(p^{\nu_p(d)}))\\
	&\qquad +\PP(d\mid A) \big(\prod_{p\in \Pa}(1-g(p^{\nu_p(d)+1})/g(p^{\nu_p(d)}))-\PP( \gcd(A_d,\prod_{p \in \Pa}p)=1)\big).
\end{align*}
The contribution of values of $d$ such that $g(p^{\nu_p(d)})=g(p^{\nu_p(d)+1})$ holds for some $p\in \Pa$ to \eqref{eq:aim} is $0$, so we may restrict our attention to values of $d$ such that  $g(p^{\nu_p(d)+1})<g(p^{\nu_p(d)})$ for all $p\in \Pa$. It follows that  \eqref{eq:aim} is bounded (using  $\max\{a+b,0\}\le |a+b|$ and the triangle inequality)  by 
\begin{align}\label{eq:2sums}
	&\le  \sum_{\substack{d\le T\\ g(d)\neq 0}} |g(d)-\PP(d\mid A)|\prod_{p\in \Pa}(1-g(p^{\nu_p(d)+1})/g(p^{\nu_p(d)})) \\
	\label{eq:2sums2}&\qquad + \sum_{\substack{d\le T\\ g(d)\neq 0\\ \forall p \in \mathcal{P}:\, g(p^{\nu_p(d)+1})<g(p^{\nu_p(d)})}} \PP(d\mid A) \big|\prod_{p \in \Pa}(1-g(p^{\nu_p(d)+1})/g(p^{\nu_p(d)}))-\PP(\gcd(A_d,\prod_{p \in \Pa}p)=1)\big|.
\end{align}
The sum \eqref{eq:2sums} can be absorbed in $D_2$. It remains to bound \eqref{eq:2sums2}, which we handle using the fundamental lemma of sieve theory. Applying Corollary \ref{cor:flprob} if $u$ is sufficiently large, and Lemma \ref{lem:upperprob} otherwise, we obtain
\begin{equation}\label{eq:flst}
	\PP( \gcd(R,\prod_{p\in \Pa}p)=1) = \prod_{p \in \Pa}(1-h(p)) (1+O(f(u))) + O\big(\sum_{\substack{e\le y^u\\ e \in \langle \Pa\rangle }} \mu^2(e) \left|\PP(e\mid R)-h(e)\right|\big)
\end{equation}
for any random variable $R$ taking values in $\NN$, and any multiplicative $h\colon \NN\to [0,1]$ vanishing outside of $\langle \Pa \rangle$, satisfying $h(p)<1$ for all $p\in \Pa$, and 
\[ \prod_{a\le p \le b} (1-h(p))^{-1} \le \left( \frac{\log b}{\log a}\right)^{\kappa}\exp\left( \frac{B}{\log a}\right)\]
for all $2\le a \le b \le y$. If $\PP(d\mid A)\neq 0$, $g(d)\neq 0$ and $g(p^{\nu_p(d)+1})<g(p^{\nu_p(d)})$ for all $p\in \Pa$, we apply \eqref{eq:flst} with $R=A_d$ and $h(n)=g(nd)/g(d)$. We may do so due to the requirement \eqref{eq:iwaniecd}. Observe $h(p)=g(p^{\nu_p(d)+1})/g(p^{\nu_p(d)})$. From \eqref{eq:flst} we find that the sum in \eqref{eq:2sums2} is bounded by
\[ \ll E_1 + E_2\]
where
\begin{align*}
	E_1&:= \sum_{\substack{d\le T \\  g(d)\neq 0}} \PP(d \mid A)\prod_{p\in \Pa}(1-g(p^{\nu_p(d)+1})/g(p^{\nu_p(d)}))f(u),\\
	E_2&:= \sum_{\substack{d\le T\\   g(d)\neq 0}} \PP(d \mid A) \sum_{\substack{ e\le y^{u}\\ e \in \langle \Pa \rangle }} \mu^2(e)|\PP(e \mid A_d)-g(de)/g(d)|.
\end{align*}
(If $\PP(d\mid A)=0$ we  define $\PP(e\mid A_d)=0$ formally). In $E_1$ and $E_2$, the parameter $u=u_d\ge 1$ may be chosen arbitrarily. We define it implicitly by $y^u=S/d$ ($u$ is at least $1$, since $S/d\ge S/T\ge y$ by our assumption on $S$ and $T$). By the triangle inequality, $E_1$ is bounded by $E_{1,1}+E_{1,2}$ where
\begin{align*}
	E_{1,1}&:= \sum_{\substack{ d\le T\\ g(d)\neq 0}} g(d)\prod_{p\in \Pa}(1-g(p^{\nu_p(d)+1})/g(p^{\nu_p(d)})) f(u),\\
	E_{1,2}&:= \sum_{\substack{ d\le T\\ g(d)\neq 0}} |\PP(d\mid A)-g(d)|\prod_{p\in \Pa}(1-g(p^{\nu_p(d)+1})/g(p^{\nu_p(d)})) f(u).
\end{align*}
The error term $E_{1,1}$ can be absorbed in $D_1$, while $E_{1,2}$ can be absorbed in $D_2$. Finally we  estimate $E_2$. By the triangle inequality, the contribution of a given pair of $d$ and $e$ to $E_2$ (when $\PP(d\mid A)\neq 0$) is at most $\mu^2(e)$ times
\begin{multline*}
	\PP(d\mid A) | \PP(e\mid A_d)-g(de)/g(d)| =\PP(d\mid A) |\PP(de \mid A)/\PP(d \mid A) - g(de)/g(d)| \\
	=|\PP(de\mid A)-(g(de)/g(d)) \PP(d\mid A)|\le |\PP(de\mid A)-g(de)| + \frac{g(de)}{g(d)}|\PP(d\mid A)-g(d)|.
\end{multline*}
Thus
\[
E_2\le  \sum_{\substack{ d\le T \\ g(d)\neq 0}} \sum_{\substack{e\le S/d\\ e \in \langle \Pa \rangle}} \mu^2(e) \bigg(|\PP(de \mid A)-g(de)|+\frac{g(de)}{g(d)}|\PP(d \mid A) -g(d)|\bigg).\]
Let us denote $de$ by $m$, so that $E_2\le E_{2,1}+E_{2,2}$ where
\begin{align*}
	E_{2,1}=\sum_{\substack{m\le S\\   m \in \langle \Pa \rangle}} |\PP(m \mid A)-g(m)| \sum_{\substack{e\mid m\\ m/e\le T}} \mu^2(e),\qquad 
	E_{2,2}=\sum_{\substack{d \le T\\ g(d)\neq 0}} |\PP(d\mid A)-g(d)| \sum_{\substack{e\le S/d \\  e \in \langle \Pa \rangle}} \mu^2(e) \frac{g(de)}{g(d)}.
\end{align*}
The inner sum in $E_{2,1}$ is trivially at most $2^{\omega(m)}$, so $E_{2,1}$ can be absorbed in $D_2$. The inner sum in $E_{2,2}$ is at most  $\prod_{p\in \Pa}(1+g(p^{\nu_p(d)+1})/g(p^{\nu_p(d)}))$ and so $E_{2,2}$ can be absorbed in $D_2$ as well.

\section{Approaches to bounding \texorpdfstring{$D_1$}{D1}}\label{sec:rankin}
Recall that $D_1$ is the first error term in the statement of Theorem \ref{thm:KM}. Below $u=\log x / \log y$ as usual.
\subsection{Elliott's approach}
Elliott \cite[Chapter 3]{Elliott1} has original approaches to bounding the sums in $D_1$.  While they do not lead to sharp bounds, they are very simple. In the discussion below we assume for simplicity that $g$ is completely multiplicative (in addition to satisfying the conditions of Theorem \ref{thm:KM}, which forces $g$ to be supported on $y$-smooth numbers). The sum over $d>T$ in $D_1$ is \[\sum_{\substack{ d>T\\ g(d)\neq 0}} g(d)\prod_{p\in \Pa}(1-g(p^{\nu_p(d)+1})/g(p^{\nu_p(d)}))=\prod_{p \le y}(1-g(p))\sum_{d>T}g(d).\]
Elliott made the beautiful observation that this sum is a  tail probability of sums of independent random variables \cite[pp.~126--127]{Elliott1}: if for every prime $p\le y$ we define independent random variables $X_{p}$ by $\PP(X_{p}=\log (p^k))=g(p^k)(1-g(p))$ ($k\in \NN_{\ge 0}$) then
\[\prod_{p \le y}(1-g(p))\sum_{d>T}g(d) = \PP\big( \sum_{p\le y} X_{p}> \log T\big).\]
Elliott developed a simple lemma \cite[Lemma 3.3]{Elliott1} to bound this tail, which is closely related to \textit{Bennett's inequality} \cite{Bennett}. Next we discuss how one can bound the sum over $d\le T$ in $D_1$. This ends up being easy to do if $S$ is much larger than $Ty$, say $S= Tx^{\delta}$ for some fixed $\delta>0$ (the interesting regime is $y=x^{o(1)}$ so $y\le x^{\delta}$ may be assumed). Since $f(u_{d})\ll_{\varepsilon} u_d^{-u_d(1-\varepsilon)}\le (\delta u)^{-\delta u(1-\varepsilon)}$ for $d\le T$ and this choice of $S$, we have
\begin{align*}
	\sum_{\substack{ d \le T\\ g(d)\neq 0}} g(d) \prod_{p\in \Pa}(1-g(p^{\nu_p(d)+1})/g(p^{\nu_p(d)})) f(u_d)&=\prod_{p \le y}(1-g(p))\sum_{\substack{d\le T\\ P(d)\le y}} g(d)f(u_d)\\
	&\ll_{\varepsilon} (\delta u)^{-\delta u(1-\varepsilon)}\prod_{p \le y}(1-g(p))\sum_{ P(d)\le y}g(d) \\
	&\le (\delta u)^{-\delta u(1-\varepsilon)} \prod_{p \le y}(1-g(p))\prod_{p \le y}\big(\sum_{i \ge 0}g(p^i)\big)\\
	&= (\delta u)^{-\delta u(1-\varepsilon)}
\end{align*}
which decays rapidly in $u$. However, ideally we would want to take $S=Ty$.
\subsection{The case \texorpdfstring{$g(d)=1/d$}{gd=1/d}}
When $\Pa$ is the set of primes up to $y$, and $g(d)=1/d$ if $d \in \langle \Pa \rangle$, there are various ways of bounding $D_1$, because then the task may be reduced to the well-understood subject of the density of $y$-smooth numbers. We discuss these approaches briefly, and focus only on the first sum in $D_1$, namely
\[ \sum_{\substack{ d>T\\ g(d)\neq 0}} g(d)\prod_{p\in \Pa}(1-g(p^{\nu_p(d)+1})/g(p^{\nu_p(d)})) =\prod_{p \le y}(1-1/p)\sum_{\substack{d>T\\ P(d)\le y}} d^{-1}\]
where $T\ge y$. One approach, used in \cite[Lemma~3.5]{FordShifted}, is integration by parts together with the bound $\sum_{d\le x,\, P(d)\le y} 1 \ll_{\varepsilon} x^{\varepsilon}+ x u^{-u}$. A second approach is  integration by parts together with the estimate $\sum_{d\le cT,\, P(d)\le y} 1\ll c^{\alpha}\sum_{d\le T,\, P(d)\le y} 1$ uniformly in $c\ge 1$ \cite[Thm.~III.5.23]{Ten}, where $\alpha\in (0,1]$  minimizes $T^{\alpha} \zeta(\alpha,y)$ where
\[\zeta(s,y) := \sum_{P(d)\le y}d^{-s} = \prod_{p \le y}(1-p^{-s})^{-1}.\]
This $\alpha$ is known as the saddle point associated with $y$-smooth numbers up to $T$. The third approach, which is easiest to generalize, is Rankin's trick:
\[\sum_{\substack{d>T\\ P(d)\le y}} d^{-1} \le  T^{\alpha-1} \sum_{P(d)\le y} d^{-\alpha}=  T^{\alpha-1}\zeta(\alpha,y)\]
holds for the same $\alpha$ as above. It is known that \cite[Thm.~III.5.21]{Ten}
\[\min_{\alpha \in (0,1]} T^{\alpha-1} \zeta(\alpha,y) \asymp \frac{\Psi(T,y)}{T} \sqrt{\frac{\log T}{\log y}\max\{1,\frac{\log T}{y}\}}\log\big(1+\frac{y}{\log T}\big).\]
\subsection{Rankin's trick}
\begin{lem}\label{lem:rankin}
Let $h\colon \NN\to \RR_{\ge 0}$ be a nonnegative multiplicative function. Suppose that for some $\kappa,B> 0$ we have
	\begin{equation}\label{eq:mertensk}
		\bigg|\sum_{p\le t} \frac{h(p)\log p}{p}-\kappa \log t\bigg|\le B
	\end{equation}
	for all $t \in [2,y]$. Suppose $x\ge y \ge 2$. If $h$ is supported on squarefrees then
	\[ \sum_{\substack{d>x\\ P(d)\le y}} \frac{h(d)}{d} \ll \prod_{p\le y}\left(1+\frac{h(p)}{p}\right) e^{O_{\kappa,B}(u)}  (u\log (u+1))^{-u}. \]
	If for all $p\le y$ and $i\ge 2$ we have $h(p^i)\le (i+1)^B$, then
	\[ \sum_{\substack{d>x\\ P(d)\le y}} \frac{h(d)}{d} \ll  \prod_{p\le y}\left(\sum_{i=0}^{\infty}\frac{h(p^i)}{p^i}\right)\cdot  \begin{cases}  (u\log (u+1))^{-u}e^{O_{\kappa,B}(u)} &\text{if }x\ge y \ge \log x,\\ x^{-1} e^{O_{\kappa,B}\big(\log \big(1+\tfrac{\log x}{y}\big)\tfrac{y}{\log y}\big)} &\text{if }\log x \ge y \ge 2.\end{cases}\]
\end{lem}
\begin{rem}
	The squarefree case of Lemma \ref{lem:rankin} was established by Tenenbaum and Wu \cite[Lem.~6.1]{TenenbaumWu} with a stronger result.\footnote{A factor of $\kappa/v$ is missing from the right-hand side of the penultimate displayed equation in the proof of \cite[Lem.~6.1]{TenenbaumWu}. Further, $\xi_{\kappa}(u)=\xi(u/\kappa)$ holds for $u\ge (e-1)\kappa$, and otherwise $\xi_{\kappa}(u)=1$. Consequently, the factor $e^{Au\xi_{\kappa}(u)/\log y}$ in the statement of \cite[Lem.~6.1]{TenenbaumWu} should  be $e^{A(u/\kappa)\xi_{\kappa}(u)/\log y}$ if $u\ge (e-1)\kappa$, and otherwise $e^{A(e-1)/\log y}$.} The non-squarefree case of Lemma \ref{lem:rankin} is covered in \cite[Lem.~6.1]{TenenbaumWu} if one restricts to $y\ge (\log x)^{2+\delta}$ for some $\delta>0$. As we are also interested in smaller $y$ we include the routine proof.
\end{rem}
\begin{proof}
	The claim is trivial if $x$ or $u$ are bounded, so we may always assume $x$ and $u$ are sufficiently large.  Under the squarefree assumption, the $d$-sum is empty if $y\le c\log x$ ($c>0$ sufficiently small), so one may assume $y \ge c\log x$ in that case. 
	By Rankin's trick, regardless of the support of $h$, we have
	\[\sum_{\substack{d>x\\ P(d)\le y}} \frac{h(d)}{d} \le x^{\alpha-1} \sum_{P(d)\le y} \frac{h(d)}{d^{\alpha}} = x^{\alpha-1}\prod_{p\le y}\left(\sum_{i=0}^{\infty}\frac{h(p^i)}{p^i}\right)\prod_{p\le y} \frac{\sum_{i=0}^{\infty}h(p^i)/p^{i\alpha}}{\sum_{i=0}^{\infty}h(p^i)/p^i}\]
	for every $\alpha \in (0,1]$. If $x\ge y\ge 2\log x$ we take $\alpha =1-\log(u\log(u+1))/\log y$, if $2\le y\le c\log x$ we take $\alpha=y/(\log x \log y)$, and in the intermediate range $y \in [c\log x, 2\log x]$ we take $\alpha = 1/\log y$. Note that
	\[ x^{\alpha-1} = \begin{cases} (u \log (u+1))^{-u} e^{O(u)} &\text{if }x \ge y \ge \log x, \\ x^{-1} e^{O(y/\log y)} &\text{if }\log x \ge y \ge 2.\end{cases}\]
	Since $\log(1+a+b)\le a +\log(1+b)$ holds for $a,b\ge 0$, we have
	\[ \log  \frac{\sum_{i=0}^{\infty}h(p^i)/p^{i\alpha}}{\sum_{i=0}^{\infty}h(p^i)/p^i}\le  \log  \frac{\sum_{i=0}^{\infty}h(p^i)/p^{i\alpha}}{1+h(p)/p}\le \frac{h(p)}{p}(p^{1-\alpha}-1)+\log(1+ \sum_{i\ge 2}h(p^i)/p^{i\alpha})\]
	and so it suffices to establish two bounds:
	\begin{align}
		\label{eq:firstbnd}
		\sum_{p\le y}  \frac{h(p)}{p}(p^{1-\alpha}-1) &\ll_{\kappa,B} \begin{cases} u &\text{if }x\ge y\ge \log x,\\ \frac{y}{\log y}&\text{if }\log x \ge y \ge 2,\end{cases}\\
		\label{eq:secondbnd} \sum_{p\le y} \log(1+\sum_{i=2}^{\infty}h(p^i)/p^{i\alpha})&\ll_{B} \begin{cases} u& \text{if }x \ge y \ge \log x, \\ 
			\log\big( 1+\frac{\log x}{y}\big)	\frac{y}{\log y} &\text{if }\log x \ge y \ge 2.\end{cases}
	\end{align}
	Estimate~\eqref{eq:firstbnd} follows from \eqref{eq:mertensk}, partial summation and $\int_{0}^t (e^s-1)ds/s \ll e^t/(t+1)$. We turn to \eqref{eq:secondbnd}. If $h$ is supported on squarefrees, \eqref{eq:secondbnd} is trivial. Otherwise, by the assumption $h(p^i)\le (i+1)^B$, the left-hand side of \eqref{eq:secondbnd} is at most 
	\begin{equation}\label{eq:finalp}
		\sum_{p\le y} \log (1+C_B(1-p^{-\alpha})^{-B-1}p^{-2\alpha}).
	\end{equation}
	If $y\ge (\log x)^3$ then $\alpha \ge 2/3+o(1)$, and \eqref{eq:finalp} is $O_B(1)$ (much smaller than needed). If $2 \le y \le 2\log x$, the $p$-th term in \eqref{eq:finalp} is $\ll_B \log(2+ (\log y \log x/(y \log p)))$ by replacing $p^{-2\alpha}$ with $1$, and $(1-p^{-\alpha})^{-1}$ with $(\alpha\log p)^{-1}$. Thus we need to show
	\[ \sum_{p\le y} \log\big(1  +\tfrac{\log y \log x}{y \log p}\big) \ll \log\big(1+\tfrac{\log x}{y}\big)\frac{y}{\log y}.\]
	The contribution of $p\le \sqrt{y}$ is $\ll \sqrt{y}\log(1+\log y\log x/y)$, which is acceptable. The contribution of $p>\sqrt{y}$ is 
	\[ \ll   \log\big(1 + \tfrac{\log x}{y}\big) \sum_{p\le y} 1\ll \log\big(1+\tfrac{\log x}{y}\big)\frac{y}{\log y}\]
	by Chebyshev's bound. Finally we treat $(\log x)^3 \ge y \ge 2\log x$. We consider separately two ranges of $p$. If $p> e^{1/\alpha}$ then $(1-p^{-\alpha})^{-B-1} \ll_B 1$ and so these primes contribute to \eqref{eq:finalp} at most
	\[ \ll_B \sum_{p\le y} p^{-2\alpha} \ll u\]
	by Chebyshev's bound. The primes $p\le e^{1/\alpha}$ contribute to \eqref{eq:finalp}, since $(1-p^{-\alpha})^{-1} \asymp \alpha^{-1}/\log p$ in this range, at most
	\[ \ll_B \sum_{p\le \min\{y,e^{1/\alpha}\}} \log\big( 2 + \frac{\alpha^{-1}}{\log p}\big)\]
	and this can be shown to be $\ll u$ using Chebyshev's bound as well.
\end{proof}
\begin{cor}\label{cor:rankin}
	Let $x\ge y\ge 2$. Under the same notation and assumptions of Lemma \ref{lem:rankin}, we have
	\[ \sum_{\substack{d\le x\\ P(d)\le y}} \frac{h(d)}{d}f\big( u+1 - \frac{\log d}{\log y}\big)\ll \prod_{p\le y}\left(1+\frac{h(p)}{p}\right) u^{-u}e^{O_{\kappa,B}(u\log_3 \max\{u,100\})}\]
for $x\ge y \ge 2$	if $h$ is supported on squarefrees, and 
	\[ \sum_{\substack{d\le x\\ P(d)\le y}} \frac{h(d)}{d}f\big( u+1 - \frac{\log d}{\log y}\big)\ll\prod_{p\le y}\left(\sum_{i=0}^{\infty}\frac{h(p^i)}{p^i}\right) \cdot \begin{cases} u^{-u}e^{O_{\kappa,B}(u\log_3 \max\{u,100\})} &\text{if }x\ge y \ge \log x, \\ x^{-1+O_{\kappa,B}(\log_3 \max\{x,100\} / \log \log x)} &\text{if }\log x \ge y \ge 2\end{cases}\]
	if $h(p^i)\le (i+1)^B$ for $i\ge 2$. Here $f$ is as in Theorem \ref{thm:KM}.
\end{cor}
\begin{proof}
We may discard the $d=1$ term. By considering $d\in (y^{\ell},y^{\ell+1}]$ for each $0\le \ell \le u$, we have
\begin{align}
\notag\sum_{\substack{1<d\le x\\ P(d)\le y}} \frac{h(d)}{d}f\big( u+1 - \frac{\log d}{\log y}\big) &\le \sum_{0\le \ell \le u} \big(\max_{d \in (y^{\ell},y^{\ell+1}]} f\big( u+1 - \frac{\log d}{\log y}\big)\big) \sum_{\substack{d>y^{\ell}\\ P(d)\le y}} \frac{h(d)}{d}\\
\label{eq:hmax}&\le (u+1) \max_{0\le \ell \le u} \big(\max_{d \in (y^{\ell},y^{\ell+1}]} f\big( u+1 - \frac{\log d}{\log y}\big)\big) \sum_{\substack{d>y^{\ell}\\ P(d)\le y}} \frac{h(d)}{d}. 
\end{align}
The $(u+1)$-factor is harmless. From properties of $f$, and Lemma~\ref{lem:rankin} with $x=y^{\ell}$, we can estimate the maximum over $d$ as well as the $d$-tail in \eqref{eq:hmax}. If $h$ is supported on squarefrees, we obtain that the right-hand side of \eqref{eq:hmax} is at most
\[  \prod_{p\le y}(\sum_{i=0}^{\infty}\frac{h(p^i)}{p^i}) e^{O_{\kappa,B}(u\log_3 \max\{u,100\})} \max_{0\le \ell \le u}((\ell+1)\log(\ell+2))^{-\ell}(u-\ell+1)^{-(u-\ell)}. \]
The maximum is of size $\ll u^{-u} e^{O(u)}$, concluding the proof in this case. If $h(p^i)\le (1+i)^B$ is assumed for $i\ge 2$, the same argument works when $y\ge \log x$. If $y<\log x$ one needs to consider $u\ge \ell > y/\log y$ separately from $\ell \le y/\log y$ because of the two ranges in Lemma~\ref{lem:rankin}. The maximum over $\ell\le y/\log y$ contributes at most $u^{-u}e^{O_{\kappa,B}(u\log_3 \max\{u,100\})}$, which is acceptable. The maximum over $u\ge \ell >y/\log y$ contributes $\ll x^{-1+O_{\kappa,B}(\log_3 \max\{x,100\} / \log \log x)}$ by a calculus computation that is left to the reader.
\end{proof}
\section{Proofs of corollaries}
\subsection{Proof of Corollary \ref{cor:giden}}
If $u$ is bounded then the corollary is trivial because $u^{-\theta u(1-\varepsilon)} \gg_{\theta,\varepsilon} 1$, so we may assume that $u$ is sufficiently large (in terms of $\theta$ and $\varepsilon$). Let $\Pa$ be the set of all primes up to $y$, so that $\langle \Pa\rangle$ is the set of $y$-smooth numbers. Let $g(n)=1/n$ for all $n\in \langle \Pa \rangle$. Then \eqref{eq:iwaniecd} holds with $\kappa=1$ by Mertens' theorem, and $(G_p)_{p\le y}$ has the same distribution as $(\Geom_p)_{p\le y}$. Theorem \ref{thm:KM} with $S=Ty$ and $T=x^{\theta}/y$ implies in this case that
\begin{equation}\label{eq:mm}
	\begin{split}
	\dtv( (\nu_p(A))_{p\le y}, (\Geom_p)_{p \le y}) &\ll_{\varepsilon} \frac{1}{\log y}\bigg(\sum_{\substack{ d>x^{\theta}/y \\ P(d)\le y}} \frac{1}{d}+\sum_{\substack{ d \le x^{\theta}/y\\ P(d)\le y}} \frac{1}{d} f(u_d)\bigg)  \\&\qquad + \sum_{\substack{d\le x^{\theta}\\ P(d)\le y }}|\PP(d\mid A)-1/d| (2^{\omega(d)}+\log y)
	\end{split}
\end{equation}
by Mertens' theorem, for any $\theta \in (0,1)$ and $\varepsilon>0$, where $u_{d}=\log(x^{\theta}/d)/\log y$. Since $\omega(d) \le \tfrac{\log d}{\log \log d} (1+o(1))$ \cite[Theorem 2.10]{MV}, the last sum in \eqref{eq:mm} is bounded by the sum in the right-hand side of \eqref{eq:dtvgeom}:
\[ \sum_{\substack{d\le x^{\theta}\\ P(d)\le y }}|\PP(d\mid A)-1/d| (2^{\omega(d)}+\log y)\ll x^{1/\log \log x}\sum_{\substack{d\le x^{\theta}\\ P(d)\le y}} |\PP(d\mid A)-1/d|.\]
It remains to bound the $1/d$-sums in \eqref{eq:mm}. This follows from Lemma \ref{lem:rankin} and Corollary \ref{cor:rankin} applied with $h\equiv 1$, and $x$ replaced by $x^{\theta}/y$.
\subsection{Proof of Corollary \ref{cor:gphi}}
Fix $a\in \ZZ\setminus \{0\}$. Let $\Pa$ be the set of all primes up to $y$, so that $\langle \Pa \rangle$ is the set of $y$-smooth numbers. For $n\in \langle \Pa \rangle$ we define $g(n)=1/\phi(n)$ if $(n,a)=1$ and $g(n)=0$ if $(n,a)>1$. Then \eqref{eq:iwaniecd} holds with $\kappa=1$ by Mertens' theorem, and $(G_p)_{p\le y}$ is distributed like $(\widetilde{G}_p)_{p\le y}$. Theorem \ref{thm:KM} with $S=Ty$ and $T=x^{\theta}/y$ implies
\begin{equation}\label{eq:dphi}
	\begin{split}
	\dtv( (\nu_p(A))_{p\le y}, (G_p)_{p \le y}) &\ll_{a,\varepsilon} \frac{1}{\log y}\bigg(\sum_{\substack{ d>x^{\theta}/y \\ P(d)\le y}} \frac{1}{\phi(d)} +\sum_{\substack{ d \le x^{\theta}/y\\ P(d)\le y}} \frac{1}{\phi(d)} u_{d}^{-u_{d}(1-\varepsilon)}\bigg) \\
	&\qquad +  \sum_{\substack{d\le x^{\theta}\\ P(d)\le y }} \left|\PP(d\mid A)-g(d)\right|(2^{\omega(d)}+\log y)
	\end{split}
\end{equation}
holds by Mertens' theorem, where  $u_{d}=\log(x^{\theta}/d)/\log y$. By definition,
\[\sum_{\substack{d\le x^{\theta}\\ P(d)\le y }} \big|\PP(d\mid A)-g(d)\big|(2^{\omega(d)}+\log y)=\sum_{\substack{d\le x^{\theta}\\ P(d)\le y\\(d,a)=1}} \big|\PP(d\mid A)-1/\phi(d)\big|(2^{\omega(d)}+\log y)+\sum_{\substack{d\le x^{\theta}\\ P(d)\le y\\(d,a)>1 }} \PP(d\mid A)(2^{\omega(d)}+\log y).\]
It remains to estimate the sums involving $1/\phi(d)$ in \eqref{eq:dphi}. This follows from Lemma \ref{lem:rankin} and Corollary \ref{cor:rankin} with $h(d)=d/\phi(d)$. See \cite[Lemma 3.5]{FordShifted} for an alternative treatment when $y\ge (\log x)^2$.
\subsection{Proof of Corollary \ref{cor:bern}}
Using Corollary \ref{cor:KM}, Lemma \ref{lem:rankin} and Corollary \ref{cor:rankin} with $h(p)/p=g(p)/(1-g(p))$, we obtain
\begin{cor}\label{cor:KMspec}
Let $A$ be a random variable  taking values in $[1,x]\cap \NN$. Let $y\in [2,x]$. Let $\Pa$ be a subset of the primes up to $y$. Let $g\colon \NN\to [0,1]$ be a multiplicative function supported on the squarefree elements in $\langle \Pa \rangle$, such that for some $\delta>0$ we have $g(p) \le 1-\delta$ for all $p\in \Pa$, and such that for some $\kappa>0$ and $B>0$, 
\[ \bigg|\sum_{p\le t} g(p)\log p-\kappa \log t \bigg|\le B, \qquad \sum_{p\le y}g(p)^2\log^2 p\le B\]
hold for $2 \le t \le y$. Given a parameter $T \ge y$, let $u_T=\log T/\log y$.  Then, for $D'_2$ as in Corollary \ref{cor:KM},
\[	\dtv( (\mathbf{1}_{p\mid A})_{p\in\Pa}, (B_p)_{p\in \Pa}) \ll_{\kappa,B,\delta} u_T^{-u_T}e^{O_{\kappa,B,\delta}(u_T\log_3 \max\{u_T,100\})}+D'_2.\]
\end{cor}
We use the notation of \S\ref{sec:ellgen}. We may assume $y\le x^{\varepsilon/6}$. Let $h\in \ZZ[t]$. Consider first the first part of Corollary \ref{cor:bern}. Let $\rho(d)$ be the number of distinct roots of $h$ modulo $d$. Recall
\[ \sum_{p \le t}\frac{\rho(p)\log p}{p} = \omega_h \log t + O_h(1)\]
where $\omega_h$ is the number of distinct irreducible factors of $h$ (to prove this, reduce to the case where $h$ is irreducible, and observe that this is Mertens' theorem for number fields). We apply Corollary \ref{cor:KMspec} to $A=h(N_x)$ with $\Pa$ being the set of primes in $(p_0,y]$, and $g(d)=\rho(d)/d$ if $d\in  \langle \Pa \rangle$ is squarefree. Since $\PP(d \mid h(N_x))-\rho(d)/d=O_{h}(\rho(d)/x)$, we obtain for $T=x^{1-\varepsilon/2}$ that
\[		\dtv( (\mathbf{1}_{p\mid h(N_x)})_{p_0<p\le y}, (W_p)_{p_0<p\le y})\ll_{h,\varepsilon}  u^{-(1-\varepsilon)u}+D'_2 \]
holds where
\[	D'_2 \ll_h  x^{-1}\sum_{\substack{d\le x^{1-\varepsilon/3}\\ P(d)\le y}}\mu^2(d)\rho(d)(2^{\omega(d)}+\prod_{p\le y}(1+\rho(p)/p))\ll x^{o(1)}\frac{\Psi(x^{1-\varepsilon/3},y)}{x} \]
using the divisor bound. To conclude recall $\Psi(x,y)\ll_{\varepsilon} xu^{-u}+x^{\varepsilon}$. The proof of the second part of Corollary \ref{cor:bern} is similar, except one estimates $D'_2$ by taking $T=x^{1/2-\varepsilon/2}$ and invoking the Bombieri--Vinogradov theorem.
\subsection*{Acknowledgments}
I thank Kevin Ford for feedback and a useful discussion concerning his work on the Kubilius model, and Brad Rodgers for feedback on an earlier version. I am grateful to the referees for their helpful and constructive comments, and in particular for bringing \cite{Brownian} to my attention. O.G.~is supported by the Israel Science Foundation (grant no.~2088/24), an Alon Fellowship and the Rabbi Dr.~Roger Herst Faculty Fellowship.
\bibliographystyle{abbrv}
\bibliography{references}

\Addresses
\end{document}